\documentclass[12pt, reqno]{amsart}

\usepackage{amsmath}
\usepackage{amsthm}
\usepackage{amssymb}
\usepackage{amsrefs}
\usepackage{mathrsfs}
\usepackage[usenames]{color}

 \newtheorem{thm}{}[section]
 \newtheorem{theorem}[thm]{Theorem}
 \newtheorem{corollary}[thm]{Corollary}
 \newtheorem{lemma}[thm]{Lemma}
 \newtheorem{proposition}[thm]{Proposition}

 \theoremstyle{definition}
 \newtheorem{definition}[thm]{Definition}
 \theoremstyle{remark}
 \newtheorem{remark}[thm]{Remark}
 
 \newtheorem{example}[thm]{Example}
 
 \numberwithin{equation}{section}
 \allowdisplaybreaks
 
\newcommand{\NN}{\ensuremath{\mathbb{N}}}
\newcommand{\XX}{\ensuremath{\mathbb{X}}}
\newcommand{\YY}{\ensuremath{\mathbb{Y}}}

\newcommand{\xx}{\ensuremath{\mathbf{x}}}
\newcommand{\yy}{\ensuremath{\mathbf{y}}}
\newcommand{\zz}{\ensuremath{\mathbf{z}}}
\newcommand{\ee}{\ensuremath{\mathbf{e}}}
\newcommand{\uu}{\ensuremath{\mathbf{u}}}
\newcommand{\sss}{\ensuremath{\mathbf{s}}}
\newcommand{\BB}{\ensuremath{\mathcal{B}}}
\newcommand{\GG}{\ensuremath{\mathcal{G}}}
\newcommand{\FF}{\ensuremath{\mathbb{F}}}
\newcommand{\FFF}{\ensuremath{\mathcal{F}}}

\newcommand{\GOW}{\operatorname{\mathrm{O}}}

\newcommand{\OO}{\operatorname{\mathcal{O}}}

\newcommand{\supp}{\operatorname{supp}}

\begin{document}

\title[Conditionality constants and superrflexivity]{Conditional quasi-greedy bases in non-superreflexive Banach spaces}

\author[F. Albiac]{Fernando Albiac}
\address{Mathematics Department\\ 
Universidad P\'ublica de Navarra\\
Campus de Arrosad\'{i}a\\
Pamplona\\ 
31006 Spain}
\email{fernando.albiac@unavarra.es}

\author[J. L. Ansorena]{Jos\'e L. Ansorena}
\address{Department of Mathematics and Computer Sciences\\
Universidad de La Rioja\\ 
Logro\~no\\
26004 Spain}
\email{joseluis.ansorena@unirioja.es}

\author[P. Wojtaszczyk]{Przemys\l{}aw Wojtaszczyk}
\address{Interdisciplinary Centre for Mathematical and Computational Modelling\\
University of Warsaw\\
 02-838 Warszawa\\
ul. Prosta 69\\
Poland}
\address{Institute of Mathematics of the Polish Academy of Sciences\\
00-656 Warszawa\\
ul. \'Sniadeckich 8\\
Poland}
\email{p.wojtaszczyk@icm.edu.pl}



\begin{abstract} For a conditional quasi-greedy basis $\BB$ in a Banach space the associated conditionality constants $k_{m}[\BB]$ verify the estimate $k_{m}[\BB]=\OO(\log m)$. Answering a question raised by Temly\-akov, Yang, and Ye, several authors have studied whether this bound can be improved when we consider quasi-greedy bases in some special class of spaces. It is known that every quasi-greedy basis in a superreflexive Banach space verifies $k_{m}[\BB]=(\log m)^{1-\epsilon}$ for some $0<\epsilon<1$, and this is optimal. Our first goal in this paper will be to fill the gap in between the general case and the superreflexive case and investigate the growth of the conditionality constants in non-superreflexive spaces. Roughly speaking, the moral will be that we can guarantee optimal bounds only for quasi-greedy bases in superreflexive spaces.  We prove that if a Banach space $\XX$ is not superreflexive then there is a quasi-greedy basis $\BB$ in a Banach space $\YY$ finitely representable in $\XX$ with $k_{m}[\BB] \approx \log m$. As a consequence we obtain that for every $2<q<\infty$ there is a Banach space $\XX$ of type $2$ and cotype $q$ possessing a quasi-greedy basis $\BB$ with 
$k_{m}[\BB] \approx \log m$. We also tackle the corresponding problem for Schauder bases and show that if a space is non-superreflexive then  it possesses a basic sequence $\BB$ with 
$k_m[\BB]\approx m$.
\end{abstract}

\thanks{F. Albiac and J.L. Ansorena were partially supported by the Spanish Research Grant \textit{An\'alisis Vectorial, Multilineal y Aplicaciones}, reference number MTM2014-53009-P. F. Albiac also acknowledges the support of Spanish Research Grant \textit{ Operators, lattices, and structure of Banach spaces},  with reference MTM2016-76808-P. P. Wojtaszczyk was partially supported by National Science Centre, Poland grant UMO-2016/21/B/ST1/00241.
}

\maketitle

\section{Introduction and background}
\label{Introduction}
\noindent 
Let $\XX$ be a Banach space over $\FF$, the real or complex scalar field. A sequence $(\xx_n)_{n=1}^\infty$ in $\XX$  is a \textit{(Schauder) basis} if 
 \begin{enumerate} 
\item[(i)] $[x_n \colon n\in\NN]=\XX$, 
\item[(ii)] there is a (unique) sequence $(\xx_n^*)_{n=1}^\infty$
in $\XX^*$ such that $\xx_n^*(\xx_k)=\delta_{k,n}$ for all $k$, $n\in\NN$, and
\item[(iii)] $K:=\sup_m \Vert S_m\Vert<\infty$,
where $S_m=S_m[\BB]=\sum_{n=1}^m \xx_n^* \otimes \xx_n$.
\end{enumerate}
The linear maps $\xx_n^*$, $n\in \NN$, are the \textit{biorthogonal functionals} associated to $(\xx_n)_{n=1}^\infty$,  the operators $S_m$, $m\in \NN,$ are the \textit{partial sum projections} of the basis and the number $K$ is the \textit{basis constant}.

A basis $(\xx_n)_{n=1}^\infty$ is said to be \textit{semi-normalized} if there are  constants $a$ and $b$ such that $0<a\le \Vert \xx_n\Vert \le b<\infty$ for all $n\in\NN$. In  case that $a=b=1$  the basis is called \textit{normalized}.
Let us recall the following simple and well known result.
\begin{lemma}[see e.g. \cite{Singer1970}*{Corollary 3.1}] \label{SemiNormalizedCharacterization} A basis $(\xx_n)_{n=1}^\infty$ 
 is semi-normalized if and only if 
\begin{equation*}
\textstyle
\sup_{n\in\NN} \max\{ \Vert \xx_n\Vert , \Vert \xx_n^*\Vert \} <\infty.
\end{equation*}
\end{lemma}
An easy consequence of  Lemma~\ref{SemiNormalizedCharacterization} is that a Banach space equipped with a semi-normalized basis $\BB$ can be renormed so that $\BB$ becomes normalized. Thus those results enunciated for normalized bases still hold true for semi-normalized bases.

Let $\BB=(\xx_n)_{n=1}^\infty$ be a semi-normalized basis of a Banach space $\XX$. By Lemma~\ref{SemiNormalizedCharacterization} we have a continuous linear operator $\FFF\colon \XX \to c_0$ given by
\[
 f \mapsto (\xx_n^*(f))_{n=1}^\infty.
\]
Hence, for any $f\in \XX$ there is an injective map $\rho\colon\NN\to\NN$ (an \textit{ordering} of $\NN$) such that 
\begin{equation}
\label{greedycondition}
|\xx_{\rho(k)}^{\ast}(f)|\ge |\xx_{\rho(n)}^{\ast}(f)| \quad\text{if}\quad k\le n.
\end{equation}
If the sequence $(\xx_n^*(f))_{n=1}^\infty$ contains several terms with the same absolute value then such an ordering is not uniquely determined. 
In order to get uniqueness we impose the additional condition
\begin{equation}
\label{desempate}
 \rho(k)\le \rho(n) \quad\text{whenever}\quad |\xx_{\rho(k)}^{\ast}(f)|= |\xx_{\rho(n)}^{\ast}(f)|.
\end{equation}
If $f$ is infinitely supported there is a unique ordering $\rho$ of $\NN$ that verifies \eqref{greedycondition} and \eqref{desempate}; moreover such an ordering verifies $\rho(\NN)=\supp (f)$. In case that $f$ is finitely supported there is a unique ordering $\rho$ that verifies \eqref{greedycondition}, \eqref{desempate}, and $\rho(\NN)=\NN$. In both cases we will refer to such a unique ordering as the \textit{greedy ordering} for $f$. 

For $m\in\NN$ let us consider the (non-linear  nor continuous) operator
\[
\textstyle
 \GG_m\colon \XX\to\XX, \quad x\mapsto \sum_{n=1}^m \xx^*_{\rho(n)}(f) \xx_{\rho(n)},
 \]
where $\rho$ is the greedy ordering for $x$. The sequence of maps $(\mathcal G_{m})_{m=1}^{\infty}$ is called the \textit{greedy algorithm} in $\XX$ associated to the basis $(\xx_{n})_{n=1}^{\infty}$. A semi-normalized basis $(\xx_n)_{n=1}^\infty$ is \textit{quasi-greedy} if and only if $\Gamma:=\sup_m \Vert \GG_m\Vert<\infty$ (see \cite{Wo2000}*{Theorem 1}). We will refer to $\Gamma$ as the \textit{quasi-greedy constant} of the basis (see \cite{AA2016, AlbiacAnsorena2017} for a detailed discussion on quasi-greedy constants).

 Given a basis $(\xx_n)_{n=1}^\infty$ in a Banach space $\XX$ and $A$ a finite subset of $\NN$, the \textit{coordinate projection} on $A$ is the linear operator
 \[
 \textstyle
S_A\colon \XX\to \XX, \quad f\mapsto \sum_{n\in A} \xx_n^*(f) \xx_n.
\]
The basis $(\xx_n)_{n=1}^\infty$ is \textit{unconditional} if and only if
 $\sup_{A \text{ finite}} \Vert S_A\Vert<\infty$ (see \cite{AlbiacKalton2016}*{Proposition 3.1.5}); otherwise $(\xx_n)_{n=1}^\infty$ is said to be \textit{conditional}. The conditionality of a basis $\BB=(\xx_n)_{n=1}^\infty$ in a Banach space $\XX$ can be measured in terms of the growth of the sequence 
\[
\textstyle
k_m :=k_m[\BB]
=\sup_{|A|\le m} \Vert S_A\Vert, \quad m=1,2,\dots.
\]
Being quasi-greedy is formally a weaker condition than being semi-normalized and unconditional. Indeed, Wojtaszczyk \cite{Wo2000} proved that in a wide class of separable Banach spaces there exist conditional quasi-greedy bases, showing this way that those two concepts are different.

In approximation theory, the sequence $(k_{m})_{m=1}^{\infty}$ is used to quantify the performance of the greedy algorithm with respect to the \textit{best $m$-term approximation error}. Indeed, if $C_{m}$ denotes the smallest number such that
\[
\Vert x-\mathcal G_{m}(x)\Vert\le C_{m}\inf\left\{\left\Vert x-\sum_{j\in A}a_{j}\xx_{j}\right\Vert : a_{j}\in \FF, |A|\le m\right\}, \quad \forall x\in \XX,
\]
then there is a close relation between $C_m$ and $k_m$. For instance,  it was proved in \cites{GHO2013,TemlyakovYangYe2011} that $C_{m}\approx k_{m}$ when $\BB$ is an \textit{almost greedy} basis (i.e., quasi-greedy and democratic). 

For every basis $\BB$ in a Banach space one always has the estimate 
$
 k_m[\BB] \lesssim m \text{ for } m\in\NN
$; this is the best one can hope for in general since there are semi-normalized bases  $\BB$, such as the summing basis of $c_{0}$, for which $ k_m[\BB] \approx m$  for $m\in\NN$. However, when the basis is quasi-greedy the size of the terms of the sequence $(k_m[\BB])_{m=1}^{\infty}$ is controlled by function that grows more slowly:
\begin{theorem}[\cite{DKK2003}*{Lemma 8.2}] \label{LogEstimate} If $\BB$ is a quasi-greedy basis in a Banach space $\XX$ then
\begin{equation}\label{newlabeleq}
k_m[\BB]
\lesssim \log m \text{ for }m\ge 2.
\end{equation}
\end{theorem}
 Temlyakov, Yang, and Ye asked in \cite{TemlyakovYangYeB} whether this bound was optimal or could be improved in special cases like Hilbert spaces. The first question was answered by Garrig\'os et al., who in \cite{GHO2013} provided examples of quasi-greedy bases for which the estimate \eqref{newlabeleq} is sharp, i.e., $k_m[\BB] \approx \log m$ for $m\ge 2$. The role played by the underlying Banach space was revealed in \cite{GW2014}, where Garrig\'os and Wojtaszczyk showed that the conditionality constants of any quasi-greedy basis for a separable $L_p$-space ($1<p<\infty$) verify the better upper estimate $k_m[\BB]\lesssim (\log m)^{\alpha} \text{ for }m\ge 2$ for some $0<\alpha<1$. Later on, Albiac et al.\  extended Garrig\'os-Wojtaszczyk's result to superreflexive Banach spaces. Recall that a Banach space $\XX$ is said to be superreflexive if any Banach space finitely representable in $\XX$ is reflexive.

 \begin{theorem}[\cite{AAGHR2015}*{Theorem 1.1}]\label{LogEstimateSuperR} For any quasi-greedy basis  $\BB$ in a superreflexive Banach space $\XX$ there is a constant $0<\alpha<1$ such that 
\begin{equation}\label{LogEstimateSuperRest}
k_m[\BB]
\lesssim (\log m)^{\alpha} \text{ for }m\ge 2.
\end{equation}
\end{theorem}

 Moreover, the estimate \eqref{LogEstimateSuperRest} cannot be improved even in Hilbert spaces. This is shown in \cite{GW2014} and relies on a general method inspired by  Olevskii for obtaining a basis for the direct sum $\XX\oplus\ell_2$ from a basis of $\XX$. We will refer to this method as the Garrig\'os-Olevskii-Wojtaszczyk method (GOW-method for short) 
and  we will denote by  $\GOW(\BB)$ the basis we obtain after applying the GOW-method to a certain basis $\BB$.

We have the following result.
 \begin{theorem}[\cite{GW2014}*{Lemma 3.4 and Lemma 3.5}]\label{TaylorGW2014} Let $\BB$ be a semi-normalized basis of a Banach space  $\XX$. Then $\GOW(\BB)=(\yy_n)_{n=1}^{\infty}$ is a quasi-greedy and democratic basis for $\XX\oplus\ell_2$ such that
  $ \Vert \sum_{n\in A} \yy_n \Vert \approx |A|^{1/2}$ for every finite set $A\subseteq\NN$.
 \end{theorem}

Continuing in the spirit of the problem raised by Temlyakov et al.\ in \cite{TemlyakovYangYeB}, at this point one may wonder if there is wider class of Banach spaces for which Theorem~\ref{LogEstimateSuperR} still holds. In \cite{AAGHR2015} it is showed that it is hopeless to try with reflexive Banach spaces. Another possibility to further research in this direction is to pay attention to the (Rademacher) type of the space. Indeed, as an immediate consequence of a theorem of Maurey and Pisier (\cite{MP1976}), every superreflexive Banach space has both non-trivial type and a non-trivial cotype. Notice also that all Banach spaces constructed in \cites{GW2014, AAGHR2015} with a quasi-greedy basis whose conditionality constant sequence is of the same order as $(\log m)_{m=2}^\infty$ are of type $1$. Thus it seems natural to investigate if Theorem~\ref{LogEstimateSuperR} can be extended to Banach spaces with non-trivial type. In this paper we solve this question in the negative:
\begin{theorem}\label{GoodTypeButBadConstants}
Let $2<q<\infty$ (respectively, $1<q<2$). There is a quasi-greedy basis $\BB$ for a Banach space $\XX$ of type $2$ and cotype $q$ (respectively, type $q$ and cotype $2$) with
\[
k_m[\BB]
 \approx \log m \text{ for }m\ge 2.
\]
\end{theorem}

Theorem~\ref{GoodTypeButBadConstants} makes us suspect  the existence of a tight connection between superreflexivity and optimal bounds for quasi-greedy bases. Of course, the right language to express this connection is using ``super-properties'', which is precisely what our next two theorems attain:
\begin{theorem}\label{CharacterizationSuperR}

A Banach space $\XX$ is non-superreflexive if and only if  there is a quasi-greedy basis $\BB$ for a
 Banach space $\YY$ finitely representable in $\XX$ with
\[
k_m[\BB] 
\approx \log m \text{ for }m\ge 2.
\]
\end{theorem}

\begin{theorem}\label{CharacterizationSuperR2}
A Banach space $\XX$ is superreflexive if and only if for every quasi-greedy basis $\BB$ of any Banach space $\YY$ finitely representable in $\XX$ there is $0<\alpha<1$ with
\[
k_m[\BB] 
\lesssim (\log m)^{\alpha} \text{ for }m\ge 2.
\]
\end{theorem}

 As a by-product of our work we will obtain a new characterization of superreflexivity that does not involve  greedy-like bases and which is of  interest by itself.
 \begin{theorem}\label{CharacterizationSuperR3}A Banach space $\XX$ is non-superreflexive if and only if there is a basic sequence $\BB$ in $\XX$ with 
 \[
 k_m[\BB]\approx m \text{ for }m\in\NN.
 \]
 \end{theorem}
 
 Theorem~\ref{CharacterizationSuperR}  and Theorem \ref{CharacterizationSuperR3} will be proved in Section~\ref{CharacterizationSection}. 
Theorem~\ref{CharacterizationSuperR2} will follow from Theorem~\ref{LogEstimateSuperR} and Theorem~\ref{CharacterizationSuperR}. In turn,
 Theorem~\ref{GoodTypeButBadConstants}
can be deduced combining Theorem~\ref{CharacterizationSuperR} with the existence of non-reflexive Banach spaces with ``good'' type and cotype from \cite{PisierXu1987}. However, this approach does not provide an explicit example of a Banach space with non-trivial type and cotype possessing a quasi-greedy basis with ``bad'' conditionality constants. In Section~\ref{PisierXu1987} we will take care of this after developing the  machinery needed in this section and the next one.

Throughout this article we follow standard Banach space terminology and notation as can be found in \cite{AlbiacKalton2016}. We single out the notation that  is more commonly employed.
As it is customary,  $[x_i \colon i\in I]$ denotes the closed linear span of the family $(x_i)_{i\in I}$, and, for $k,n\in\NN$, $\delta_{k,n}=1$ if $n=k$ and $\delta_{k,n}=0$ otherwise. Also, for $x\in\XX^*$ and $x\in\XX$, $x^*\otimes x$ denotes the endomorphism of $\XX$ given by $f\mapsto x^*(f) x$. Given families of positive real numbers $(\alpha_i)_{i\in I}$ and $(\beta_i)_{i\in I}$, the symbol $\alpha_i\lesssim \beta_i$ for $i\in I$ means that $\sup_{i\in I}\alpha_i/\beta_i <\infty$, while $\alpha_i\approx \beta_i$ for $i\in I$ means that $\alpha_i\lesssim \beta_i$ and $\beta_i\lesssim \alpha_i$ for $i\in I$. Applied to Banach spaces, the symbol $\XX\approx \YY$ means that the spaces $\XX$ and $\YY$ are isomorphic. We write $\XX\oplus\YY$ for the Cartesian product of the Banach spaces $\XX$ and $\YY$ endowed with the norm 
\[
\Vert (x,y)\Vert=\max\{ \Vert x\Vert, \Vert y\Vert\}, \quad x\in\XX,\ y\in\YY. 
 \]
 The \textit{support} of a vector $f\in\XX$ with respect to a basis $(\xx_n)_{n=1}^\infty$ with biorthogonal functionals $(\xx_n^*)_{n=1}^\infty$ is the set
 \[
 \supp(f)=\{ n\in\NN \colon \xx_n^*(f)\not=0\}.
 \]
Other more specific notation will be specified in context when needed.

\section{Tailoring bases with large conditionality constants}
\label{PisierXu1987}

\noindent We start this section introducing two types of bases whose definition goes back to Singer \cite{Singer1970}.
\begin{definition} Let $(\xx_n)_{n=1}^\infty$ be a basis in a Banach space $\XX$.
\begin{enumerate}
\item[(i)] $(\xx_n)_{n=1}^\infty$ is said to be of \textit{type P} if there are positive constants $b,C$ such that $ \Vert x_n\Vert\ge b$ and $
\Vert \sum_{k=1}^n \xx_k\Vert\le C$  for all $n$.
\item[(ii)] $(\xx_n)_{n=1}^\infty$ is said to be of \textit{type P*} if the sequence of its  biorthogonal functionals   is basis of type P for its closed linear span in $\XX^{\ast}$.
\end{enumerate}
\end{definition}

With the aim of being as self-contained as possible we gather a few properties of this kind of bases in the next lemma. 

\begin{lemma}[see \cite{Singer1970}*{Chapter 2, \S9}]\label{PProperty}Let $\BB=(\xx_n)_{n=1}^\infty$ be a basis. Then:\begin{itemize}

\item[(a)] $\BB$ is of type P if and only if it is semi-normalized and
 $\BB_0=(\sum_{k=1}^n\xx_k)_{n=1}^\infty$ is a basis, in which case $\BB_0$ is of type P*.
 
 \item[(b)] $\BB$ is of type P* if and only if it is semi-normalized and $\BB_1=(\xx_n-\xx_{n-1})_{n=1}^\infty$ (with the convention $\xx_0=0$)   is a basis,  in which case $\BB_1$ is of type P.

 \item[(c)] $\BB$ is of type P* if and only if $\sup_n \Vert \xx_n\Vert<\infty$ and there is constant $C$ such that   $|\sum_{n=1}^\infty a_n|\le C \Vert \sum_{n=1}^\infty a_n \xx_n\Vert$ 
for any sequence of scalars $(a_n)_{n=1}^\infty$ eventually zero.
\end{itemize}
Moreover, all the constants related to  bases in the conclusions of parts (a), (b), and (c) depend only on the constants related to the bases in the respective hypotheses. 
\end{lemma}

Although it may have gone unnoticed, as a matter of fact reflexivity can be characterized in terms of bases of type P and type P*. To be precise:
\begin{theorem}[cf.\ \cite{Pisier2016}]\label{CharacterizationReflexivity} If $\XX$ is a Banach space, the following are equivalent:
\begin{itemize}
\item[(a)] $\XX$ is  not reflexive.
\item[(b)] $\XX$ contains a basic sequence of type P.
\item[(c)] $\XX$ contains a basic sequence of type P*.
\end{itemize}
Moreover, all the constants related to the basic sequences in (b) and (c) (namely, the basis constants, the type P and type P* constants, and the semi-normalization constants) are universal and do not depend on the  space $\XX$.
\end{theorem}

\begin{proof} (a) $\Rightarrow$ (c) Assume that $\XX$ is not reflexive and pick $0<\theta<1$. By \cite{Pisier2016}*{Theorem 3.10 and Remark 3.21} there are sequences $\BB=(\xx_n)_{n=1}^\infty$ in $\XX$
and $(\xx_n^*)_{n=1}^\infty$ in $\XX^*$ so that both $\BB$ and $(\xx_n-\xx_{n-1})_{n=1}^\infty$ are basic sequences (with basis constants as close to 4 as we wish), 
$\sup_{n}\max\{\Vert \xx_n\Vert, \Vert \xx_n^*\Vert\}\le 1$, and $\xx_n^*(\xx_n)=\theta$. In particular,  $\theta\le \inf_n \Vert \xx_n \Vert $, hence
$\BB$ is semi-normalized. By Lemma~\ref{PProperty}(b), $\BB$ is a basic sequence of type P*.

(c) $\Rightarrow$ (b) is straightforward by Lemma~\ref{PProperty}.

(b) $\Rightarrow$ (a) Assume that $\BB=(\xx_{n})_{n=1}^{\infty}$ is a basis of type P with basis constant $K$ of a subspace of $\XX$. Let $\yy_n=\sum_{j=1}^n \xx_j$.  If a subsequence of $(\yy_n)_{n=1}^\infty$ were weakly convergent to some $z\in\XX$ then we would have $x_n^*(z)=1$ for $n=1,2,\dots$, which is impossible. Since  $(\yy_n)_{n=1}^\infty$ is bounded,  this shows that $\XX$ is non-reflexive because its unit ball is not weakly compact.
\end{proof}

The following alteration of a basis will be also be very useful to us.

\begin{lemma}
\label{Twisting} 
Let $(\xx_n)_{n=1}^\infty$ be a semi-normalized basis for a Banach space $\XX$. Then the sequence $\BB=(\yy_n)_{n=1}^\infty$ given by
\[
\begin{cases}\yy_{2n-1}=\xx_{2n-1}+\xx_{2n},\\ \yy_{2n}=\xx_{2n-1}-\xx_{2n},
\end{cases}
\]
is a semi-normalized basis for $\XX$.
\end{lemma}

\begin{proof}
By Lemma~\ref{SemiNormalizedCharacterization},
$C:=\sup_n \{ \Vert \xx_n\Vert , \Vert \xx_n^*\Vert \}<\infty$. It is clear that $[\yy_n \colon n\in\NN]=\XX$.
Define $(\yy_n^*)_{n=1}^\infty$ by 
\[
\begin{cases}\displaystyle
\yy_{2n-1}^*=\frac{\xx^*_{2n-1}+\xx^*_{2n}}{2},\\ \displaystyle \yy_{2n}^*=\frac{\xx^*_{2n-1}-\xx^*_{2n}}{2},\end{cases}
\]
so that $\yy_n^*(\yy_k)=\delta_{k,n}$ for all $k$ and $n\in\NN$. For $m\in\NN$
we have
\[
\begin{cases}
S_{2m-1}[\BB] 
=S_{2m-2}
+\frac{1}{2} \sum_{\varepsilon,\delta\in\{0,1\}} \xx_{2m-\varepsilon}^*\oplus \xx_{2m-\delta},\\
S_{2m}[\BB]=S_{2m}.
\end{cases}
\]
We infer that 
\[
\textstyle
\sup_m \Vert S_m[\BB]\Vert \le K+2C^2.\] Hence $\BB$ is a basis.
Moreover 
$ \Vert \yy_k \Vert \le 2 C$ and $ \Vert \yy_k^* \Vert \le C$
for every $k\in\NN$, and so $\BB$ is semi-normalized.
\end{proof}

Let us introduce an auxiliary measure of the conditionality of a basis $\BB$. For any $m\in\NN$ put
\begin{equation}\label{NewConditionalConstants}
L_m[\BB]=\sup\left\{ \frac{ \Vert S_A(f)\Vert}{ \Vert f \Vert } \colon \max(\supp(f))\le m, \, A\subseteq\NN\right\}.
\end{equation}

Obviously, $L_m[\BB]\le k_m[\BB]$. Notice that the basis $\BB$ is unconditional if and only if $\sup_m L_m[\BB]<\infty$.

Before giving our first result on bases, it will be convenient to fix some more notation. 

\begin{definition} Given sequences $\BB_1=(\xx_n)_{n=1}^\infty$ and $\BB_2=(\yy_n)_{n=1}^\infty$ in Banach spaces $\XX$ and $\YY$, respectively, let us define a sequence $\BB_1 \diamond \BB_2=(\zz_n)_{n=1}^\infty$ in $\XX\oplus\YY$ by 
\begin{equation*}
\begin{cases}
\zz_{2n-1}=(\xx_n,\yy_n),\\
 \zz_{2n}=(\xx_n, -\yy_n).\end{cases}
\end{equation*}
\end{definition}
\begin{theorem}\label{BasisInDirectSum}
Suppose $\BB_1$, $\BB_2$ are  semi-normalized bases in  $\XX$, $\YY$, respectively. Then the sequence $\BB_1\diamond \BB_2$ is a semi-normalized basis for $\XX\oplus\YY$. Moreover,  if $\BB_1$ is of type P and $\BB_2$ is of type P*, then 
\[
L_m[\BB_1\diamond\BB_2]\approx m \text{ for }m\in\NN.
\]
\end{theorem}

\begin{proof}Let $\BB_1=(\xx_n)_{n=1}^\infty$ and $\BB_2=(\yy_n)_{n=1}^\infty$. By \cite{Singer1970}*{Proposition 4.2} the sequence $(\uu_n)_{n=1}^\infty$ given by $\uu_{2n-1}=(\xx_n,0)$ and $ \uu_{2n}=(0,\yy_n)$, $n\in\NN$, is a semi-normalized basis for $\XX\oplus\YY$. Applying Lemma~\ref{Twisting} to $(\uu_n)_{n=1}^\infty$ we get that $\BB_1\diamond \BB_2=(\zz_n)_{n=1}^\infty$ is  semi-normalized. 

Assume now that  $\BB_1$ is of type P and that $\BB_2$ is of type P*. Put $C=\sup_n \Vert \sum_{k=1}^n \xx_k\Vert$. By Lemma~\ref{PProperty}(c) there is constant $D>0$ such that $\Vert \sum_{n=1}^m \yy_n\Vert\ge D m$ for all $m\in\NN$.
For $m\in\NN$ with $m\ge 2$ we pick $j\in\NN$ so that $2j\le m \le 2j+1$ and consider
\begin{align*}
f&:=\textstyle\sum_{n=1}^{2j} \zz_n=\left(2\sum_{n=1}^j \xx_n,0\right)=2(\yy_n,0), \text{ and }\\
g&:=\textstyle\sum_{n=1}^{2j} (-1)^{n-1} \zz_n=\left(0,2\sum_{n=1}^j \yy_n\right).\\
\end{align*}
Notice that $f+g=2 S_A(f)$, where $A=\{2 n-1\colon 1\le n\le j\}$. Hence
\[
L_{m}[\BB_1\diamond\BB_2]\ge \frac{\Vert S_A(f)\Vert }{\Vert f\Vert} \ge \frac{\Vert g \Vert -\Vert f\Vert } {2 \Vert f\Vert}\ge \frac{D j -C}{2C}\ge \frac{D(m-1)-2C}{4C},
\]
and so $L_m[\BB_1\diamond\BB_2]\approx m$ for $m\in\NN$. 
\end{proof}

\begin{remark} Note the that naive alteration of a basis described in Lemma~\ref{Twisting} can produce a not so naive alteration in its  conditional constants. Indeed, let $\BB_1$ be the canonical basis of $c_0$ (which is, up to equivalence, the unique unconditional basis with property P) and let $\BB_2$ be the canonical basis of $\ell_1$ 
 (which is, up to equivalence, the unique unconditional basis with property P*). Then, while the direct sum of $\BB_1$ and $\BB_{2}$ is still unconditional, the basis $\BB_1\diamond\BB_2$ is as conditional as it could be!

\end{remark}

Garrig\'os and Wojtaszczyk  \cite{GW2014} measured  the conditionality of the bases that are manufacture by the GOW-method.
The technique they developed allows to express the conditionality of a basis  $\GOW(\BB)$ in terms of the numbers $L_m[\BB]$. Recall that a function $\phi\colon[0,\infty)\to[0,\infty)$  is said to be \textit{doubling} if for some non-negative constant $C$ one has
$\phi(2t)\le C\phi(t)$ for all $t\ge 0$.
\begin{theorem}[\cite{GW2014}*{Lemma  3.6}]\label{ConditionalityGOW} Let  $\phi\colon[0,\infty)\to[0,\infty)$ be an increasing  doubling function. Suppose  that $\BB$ is a basis with
 $L_m[\BB] \ge \phi(m) $ for $m\in\NN$. Then $L_m[\GOW(\BB)] \ge C \phi(\log m)$ for all $m\in\NN$, where $C$ is a constant that depends only  on the function
$\phi$. 
\end{theorem} 

\begin{proof}The proof of \cite{GW2014}*{Lemma  3.6} reveals that 
$L_{m_j}[\GOW(\BB)]\ge L_j[\BB]$, where $j\in\NN$ and  $m_j=\sum_{k=1}^j 2^k=2^{j+1}-2$.  For $m\ge 2$, pick $j\in\NN$
such that $2^{j+1}-2\le m \le 2^{j+2}-3$. We have
\begin{align*}
L_m[\GOW(\BB)]
&\ge L_{m_j}[\GOW(\BB)] \ge \phi(j)
\ge \phi(\log_2((m+3)/4))
\ge \phi(\log_2(m)/4)\\
&\ge C \phi(\log(m)),
\end{align*}
where  $C$ depends on the doubling constant of $\phi$.
\end{proof}

\begin{theorem}\label{TailoringQGB}
Suppose $(\xx_n)_{n=1}^\infty$ is a basis of type P in a Banach space $\XX$.
Then the GOW-method applied to the sequence 
\[
\textstyle
(\xx_n)_{n=1}^\infty\diamond (\sum_{k=1}^n \xx_k)_{n=1}^\infty
\]
 yields a quasi-greedy basis $\BB=(\yy_n)_{n=1}^\infty$ for $\XX\oplus\XX\oplus\ell_2$ with 
\begin{itemize}
\item[(i)] $k_m[\BB]\approx L_m[\BB]\approx \log m\; \text{for}\; m\ge 2$, and
\item[(ii)] $\Vert \sum_{n\in A} \yy_n\Vert \approx |A|^{1/2}$ for any $A\subseteq\NN$ finite.
\end{itemize}
\end{theorem}

\begin{proof}By Lemma~\ref{PProperty}(a), the sequence $(\sum_{k=1}^n \xx_k)_{n=1}^\infty$ is a basis of type P* for $\XX$. Now, the result follows by combining Theorem~\ref{BasisInDirectSum},
Theorem~\ref{ConditionalityGOW}, and Theorem~\ref{TaylorGW2014}.
\end{proof} 

In light of Theorem~\ref{TailoringQGB}, in order to build a basis as in Theorem~\ref{GoodTypeButBadConstants} the remaining ingredient is a basis of type P in a Banach space of suitable type and cotype. Our construction of such a basis will rely  on a construction by Pisier and Xu \cite{PisierXu1987} of non-reflexive Banach spaces of non-trivial type and a non-trival cotype, based on the real interpolation method. 
 
Given a \textit{compatible couple} $(\XX_0,\XX_1)$ of Banach spaces, the \textit{real interpolation space} of indices $0<\theta<1$ and $1\le q<\infty$ is defined by
\[
(\XX_0,\XX_1)_{\theta,q}=\left\{f\in\XX_0+\XX_1 \colon \Vert f\Vert_{\theta,q}^q=
\int_0^\infty K^q(f,t,\XX_0,\XX_1) \frac{dt}{t^{1+ \theta q}} <\infty \right\}
\]
where, for $t>0$,
\[
K(f,t,\XX_0,\XX_1)=\inf \{ \Vert f_0\Vert_{\XX_0} + t \Vert f_1\Vert_{\XX_1} \colon f_0\in\XX_0,f_1\in\XX_1, f=f_0+ f_1\}.
\]

Real interpolation provides an exact interpolation scheme, i.e., if $T\colon (\XX_0,\XX_1) \to (\YY_0,\YY_1)$ is an admissible operator between two compatible couples (that is, $T$ is linear from $\XX_0+\XX_1$ into $\YY_0+\YY_1$ and bounded from $\XX_i$ into $\YY_i$, $i=0,1$) then
\[
\textstyle
\Vert T\colon (\XX_0,\XX_1)_{\theta,q} \to (\YY_0,\YY_1)_{\theta,q}\Vert 
\le \max_{i=0,1} \Vert T\colon \XX_i\to \YY_i\Vert.
\]
Here, at it is customary, $\Vert T\colon \XX\to\YY\Vert$ denotes the norm of the operator $T$ when regarded as an operator from $\XX$ into $\YY$. In particular we have that
\[
\XX_0\cap \XX_1 \subseteq (\XX_0,\XX_1)_{\theta,q} \subseteq \XX_0+\XX_1,
\]
with norm-one inclusions. Recall that the norms in $\XX_0\cap \XX_1$ and $\XX_0+\XX_1$ are respectively given by 
\[
\Vert f \Vert_{\XX_0\cap\XX_1}=\max\{\Vert f\Vert_{\XX_0},\Vert f\Vert_{\XX_1}\}, \]
and
\[
\Vert f \Vert_{\XX_0 + \XX_1}=K(f,1,\XX_0,\XX_1).
\]

The behavior of bases with respect to the real interpolation method is described in the next proposition.
\begin{proposition}
\label{InterpolatingBases} 
Let $(\XX_0,\XX_1)$ be a compatible couple of Banach spaces and let $\BB=(\xx_n)_{n=1}^\infty$ be a sequence in $\XX_0\cap \XX_1$ such that $[\xx_n \colon n\in\NN]=\XX_0\cap \XX_1$.  Assume that $\BB$ is a basis for $\XX_i$, $i=0,1$. Then $\BB$ is a basis for $(\XX_0,\XX_1)_{\theta,q}$ for every $0<\theta<1$ and every $ 1\le q<\infty$. Moreover:
\begin{itemize}
\item[(a)] If $\BB$ is semi-normalized in both $\XX_0$ and $\XX_1$, then $\BB$ is semi-normalized in $(\XX_0,\XX_1)_{\theta,q}$.
\item[(b)] If $\BB$ is of type P in both $\XX_0$ and $\XX_1$, then $\BB$ is of type P in $(\XX_0,\XX_1)_{\theta,q}$.
\end{itemize}
\end{proposition}

\begin{proof} For $i=0,1$ let $(\xx_n^{*,i})_{n=1}^\infty$ and $K_{i}$ be, respectively, the biorthogonal functionals and the basis constant of $\BB$ in  $\XX_i$. For $n\in\NN$ we have that $\xx_n^{*,0}(\xx_k)= \xx_n^{*,1}(\xx_k)$ for any $k\in\NN$ and that $\xx_n^{*,i}$ is a continuous map in the topology of $\XX_0\cap\XX_1$, $i=0,1$. Thus $\xx_n^{*,0}(f)= \xx_n^{*,1}(f)$ for all $f\in \XX_0\cap\XX_1$ and so there is an admissible operator $\xx_n^*\colon (\XX_0,\XX_1) \to (\FF,\FF)$ that verifies $\xx_n^{*}(\xx_k)=\delta_{k,n}$ for every $k\in\NN$. Therefore, for $m\in\NN$ we can safely define an admissible operator $S_m\colon (\XX_0,\XX_1) \to (\XX_0,\XX_1)$ by $S_m(f)=\sum_{n=1}^m \xx_n^*(f) \xx_n$. By interpolation,
\begin{align*}
\Vert S_m\colon (\XX_0,\XX_1)_{\theta,q} \to (\XX_0,\XX_1)_{\theta,q}\Vert &
\le \max_{i=0,1} \Vert S_m\colon \XX_i\to \XX_i\Vert
\\
&\le \max\{ K_0,K_1\}.
\end{align*}
Since the linear span of $\BB$ is dense in $\XX_0\cap\XX_1$, and $\XX_0\cap\XX_1$ is dense in $(\XX_0,\XX_1)_{\theta,q}$ (see \cite{BenSha1988}*{Theorem 2.9}) we infer that $\BB$ is a basis for $(\XX_0,\XX_1)_{\theta,q}$.

Assume now that $\BB$ is semi-normalized in $\XX_i$, $i=0,1$. Then, by Lemma~\ref{SemiNormalizedCharacterization},
\begin{equation*}
\textstyle
C=
\sup_{n\in\NN, i=0,1} \{ \Vert \xx_n\Vert_{\XX_i} , \Vert \xx_n^* \colon \XX_i\to \FF\Vert \} <\infty.
\end{equation*}
Hence, for all $n\in\NN$,
\[
\Vert \xx_n^* \colon (\XX_0,\XX_1)_{\theta,q}\to\FF\Vert \le \max_{i=0,1} \Vert \xx_n^* \colon \XX_i \to \FF \Vert \le C,
\]
and 
\[
\Vert \xx_n\Vert_{\theta,q} \le \Vert \xx_n\Vert_{\XX_0\cap\XX_1} \le C.
\]
By Lemma~\ref{SemiNormalizedCharacterization}, the basis is semi-nomalized in $(\XX_0,\XX_1)_{\theta,q}$.

Finally, assume that $\BB$ is of type P in $\XX_i$, $i=0,1$. Let $\BB_0=(\sum_{k=1}^n \xx_k)_{n=1}^\infty$. By Lemma~\ref{PProperty}(a) and Lemma~\ref{PProperty}(b), both $\BB$ and $\BB_0$ are semi-normalized bases for $\XX_i$, $i=0,1$. Hence, by the already proved part of this proposition, both $\BB$ and $\BB_0$ are semi-normalized bases for   $(\XX_0,\XX_1)_{\theta,q}$. From Lemma~\ref{PProperty}(a), $\BB$ is of type P in the interpolated space.
 \end{proof}

 Let $v_1$ be the space of all sequences of scalars of bounded variation, i.e.,
\[
\textstyle
v_1=\{ f=(a_n)_{n=1}^\infty \colon \Vert f \Vert_{v_1}= |a_1|+\sum_{n=1}^\infty 
|a_{n+1}-a_n|<\infty\}.
\]
The Banach space $v_1$ is nothing but $\ell_1$ in a rotated position. Indeed, the linear bijection
\begin{equation}\label{MapQ}
\textstyle
Q\colon \FF^\NN \to \FF^\NN, \quad (a_n)_{n=1}^\infty\mapsto \left(\sum_{k=1}^n a_k\right)_{n=1}^\infty
\end{equation}
restricts to an isometry from $\ell_1$ onto $v_1$. Via this isometry the functional on $\ell_1$ given by $ (a_n)_{n=1}^\infty\mapsto\sum_{n=1}^\infty a_n$ induces a functional on $v_1$ given by $ (a_n)_{n=1}^\infty\mapsto\lim_n a_n$. Therefore
\[
v_1\subseteq c:=\{ (a_n)_{n=1}^\infty \colon \exists \lim_n a_n\in\FF\}\subseteq\ell_\infty.
\]
Let $v_1^0$ be the subspace of codimension one of $v_1$ that corresponds to
\[
\textstyle
\ell_1^0=\{ (a_n)_{n=1}^\infty \in \ell_1 \colon \sum_{n=1}^\infty a_n=0\}
\]
under the isometry $Q$, i.e., 
\[
v_1^0=v_1\cap c_0= \{ (a_n)_{n=1}^\infty \in v_1 \colon \lim_n a_n=0\}.
\]

Consider also the linear mapping
\begin{equation}\label{MapR}
R\colon \FF^\NN \to \FF^\NN,
\quad
 (a_n)_{n=1}^\infty \mapsto (a_{n+1}-a_1)_{n=1}^\infty.
 \end{equation}
 It is clear that $R$ restricts to an isomorphism from $c_0$ onto $c$ with inverse
 \begin{equation}\label{MapT}
 T\colon c \to c_0,
\quad
 (a_n)_{n=1}^\infty\mapsto (a_{n-1}-\lim_n a_n)_{n=1}^\infty \, (\text{where $a_0=0$}).
 \end{equation}
 It is not hard to see that $R$ restricts as well to and isomorphism from $v_1^0$ onto $v_1$. To realize that we can use that the lifting
 \begin{equation}\label{MapL}
 L\colon \FF^\NN\to \FF^\NN, \quad
 (a_n)_{n=1}^\infty\mapsto (a_{n+1})_{n=1}^\infty
 \end{equation} defines an isomorphism from
$\ell_1^0$ onto $\ell_1$ whose inverse is given by 
$(b_n)_{n=1}^\infty\mapsto ( -\sum_{n=1}^\infty b_n, b_1, b_2,\dots)$, and that $Q\circ L=R\circ Q$.

\bigskip
 Pisier and Xu \cite{PisierXu1987} investigated the interpolated spaces $(v_1,\ell_\infty)_{\theta,q}$ for  $0<\theta<1$ and $1\le q<\infty$. However, the space $(v_1^0,c_0)_{\theta,q}$ is more fit for our purposes.
 
\begin{lemma}[cf.\ \cite{Pisier2016}*{Section 12.2}] Let $0<\theta<1$ and $1\le q<\infty$. Then:
\begin{itemize}

\item[(a)] The space $(v_1,\ell_\infty)_{\theta,q}=(v_1,c)_{\theta,q}$ with equality of norms.

\item[(b)] The space $(v_1^0,c_0)_{\theta,q}$ is isomorphic to $(v_1,\ell_\infty)_{\theta,q}$.

\item[(c)]  $(v_1^0,c_0)_{\theta,q}$ is a subspace of codimension one of $(v_1,\ell_\infty)_{\theta,q}$. 

\end{itemize}
\end{lemma}

\begin{proof}
In order to prove (a) it suffices to show that $(v_1,\ell_\infty)_{\theta,q}\subseteq c$. Let  $f=(a_n)_{n=1}^\infty\in\ell_\infty\setminus c$. There exists $\varepsilon>0$ such that for every $j\in\NN$ there are $n\ge k\ge j$ verifying $|a_n-a_k|\ge\varepsilon$. Let $f=g+h$ with $g=(b_n)_{n=1}^\infty\in v_1$ and  $h=(c_n)_{n=1}^\infty\in \ell_\infty$. Since $g\in c$, there is $j\in\NN$ such that $|b_n-b_k|\le \varepsilon/2$ for every $n\ge k\ge j$. We infer that  $|c_n-c_k|\ge \varepsilon/2$ for some $n\ge k\ge j$. Therefore $\Vert h\Vert_\infty\ge \varepsilon/4$. Consequently, $K(f,t,v_1,\ell_\infty)\ge t \varepsilon/4$ for every $t>0$, thus $\Vert f\Vert_{\theta,q}=\infty$.

 It follows by interpolation that the mapping $R$  in \eqref{MapR} defines an isomorphism from $(v_1^0,c_0)_{\theta,q}$ onto $ (v_1,c)_{\theta,q}$ whose inverse is a restriction of the operator $T$ given in \eqref{MapT}. Hence, (b) holds.

Let $P\colon \FF^\NN \to\FF$ be the projection onto the first coordinate and $L$ be as in \eqref{MapL}. By interpolation, the map  $(P,L)$ is an isomorphism from $ (v_1^0,c_0)_{\theta,q}$ onto $\FF\oplus (v_1^0,c_0)_{\theta,q}$. Hence, $ (P,L)\circ T=(P\circ T, L\circ T)$ is an isomorphism from $(v_1,c)_{\theta,q}$ onto $\FF\oplus (v_1^0,c_0)_{\theta,q}$. The proof is over by checking that $L\circ T$ is the identity map on $c_0$.
\end{proof}

Hereinafter the \textit{canonical basis} of $\FF^\NN$ will be denoted by $(\ee_n)_{n=1}^\infty$, i.e., $\ee_n=(\delta_{k,n})_{k=1}^\infty$, while $(\sss_n)_{n=1}^\infty$ will denote the \textit{summing basis} of $\FF^\NN$, i.e., $\sss_n=\sum_{k=1}^n \ee_k$.

\begin{proposition}[cf.\ \cite{Pisier2016}*{Proposition 12.7}]\label{BasisPisierXu} Let $0<\theta<1$ and $1\le q<\infty$. Then $(\ee_n)_{n=1}^\infty$ is a basis of type P for $(v_1^0,c_0)_{\theta,q}$.
\end{proposition}

\begin{proof} It is clear and well known that $(\ee_n)_{n=1}^\infty$ is a basis of type P for $c_0$. Dualizing we see that $(\ee_n)_{n=1}^\infty$ is a basis of type P* for $\ell_1$. Let $Q$ and $T$ be as in \eqref{MapQ} and \eqref{MapT} respectively. We have that $T\circ Q$ is an isomorphism from $\ell_1$ onto $v_1^0$ and that  $T\circ Q(\ee_n)=-\sss_n$. Hence  $(\sss_n)_{n=1}^\infty$ is a basis of type P* for $v_1^0$.  By Lemma~\ref{PProperty}(b), $(\ee_n)_{n=1}^\infty$ is a basis of type P for $v_1^0$. We complete the proof by an appeal to Proposition~\ref{InterpolatingBases}(c).
\end{proof}

\begin{example} Let $0<\theta<1$, $1\le q<\infty$, and $\XX= (v_1^0,c_0)_{\theta,q}$.   If $2<(1-\theta)^{-1}\le q$ then $\XX$ (which is isomorphic to $(v_1,\ell_\infty)_{\theta,q}$) has type $2$ and cotype $q$, while if $q\le (1-\theta)^{-1} <2$ then $\XX$ has type $q$ and cotype $2$ (see \cite{PisierXu1987}*{Theorem 1.1}). Obviously, $\XX\oplus\XX\oplus\ell_2$ is of the same type and cotype as $\XX$. Finally, by Theorem~\ref{TailoringQGB} and Proposition~\ref{BasisPisierXu},  the GOW-method applied to $(\ee_n)_{n=1}^\infty\diamond (\sss_n)_{n=1}^\infty$ yields a quasi-greedy basis $\BB$ for $\XX\oplus\XX\oplus\ell_2$ with $k_m[\BB]\approx L_m[\BB] \approx \log m$ for $m\ge 2$.
\end{example}

\section{New characterizations of Superreflexivity}
\label{CharacterizationSection}
\noindent
In this section we will repeatedly  use  the following refinement of the technique used by Mazur for solving the basic sequence problem.
\begin{lemma}\label{BasicSequenceLemma} Let $\varepsilon>0$ and $\XX_0$ be a finite-dimensional subspace of a Banach space $\XX$. There is a finite co-dimensional subspace $\YY\subseteq\XX$ such that
\[
\Vert f\Vert \le (1+\varepsilon) \Vert f+ g\Vert, \quad f\in\XX_0, \, g\in\YY.
\]
\end{lemma}
\begin{proof}
Although not explicitly stated in this form, the proof of  \cite{LinTza1977}*{Lemma 1.a.6} yields this result.
\end{proof}

The longed for Theorem~\ref{CharacterizationSuperR} is now a part of the following theorem.

\begin{theorem}\label{MainTheorem} Let $\XX$ be a Banach space. The following are equivalent:

\begin{itemize}

\item[(a)] $\XX$ is not superreflexive.

\item[(b)] There is a Banach space finitely representable in $\XX$ possessing a basis $\BB$ with $L_m[\BB]\approx m$ for $m\in\NN$.

\item[(c)] There is a Banach space finitely representable in $\XX$ possessing a quasi-greedy basis $\BB$ with 
 $L_m[\BB]\approx \log m$ for $m\ge 2$.

\item[(d)] There is a Banach space finitely representable in $\XX$ possessing a quasi-greedy basis $\BB$ with 
 $k_m[\BB]\approx \log m$ for $m\ge 2$.

\end{itemize}

\end{theorem}

\begin{proof} 
 
The method for proving (a) $\Rightarrow$ (b) is a refinement of that used for proving Theorem~\ref{BasisInDirectSum}. The underlying idea behind the construction below is to obtain a basis containing arbitrarily large blocks whose first half is a (finite) basic sequence of type P and whose second half is a basic sequence of type P*. Then, we pair every vector in the first half of the block with a vector in the second half of the block and, as in Lemma~\ref{Twisting}, we replace these vectors with their sum and their difference.

By Theorem~\ref{CharacterizationReflexivity} there is  a basis $\BB=(\xx_n)_{n=1}^\infty$ of type P for a Banach space $\YY$ finitely representable in $\XX$. Let $(\xx_n^*)_{n=1}^\infty$ be its biorthogonal sequence and  $(S_m)_{m=1}^\infty$ be  its partial sum projections. Define $\BB_1=(\yy_n)_{n=1}^\infty$ in $\YY$ by
\[
\yy_n=
\begin{cases}
\xx_n & \text{ if } 2^{j}-1 \le n\le 3\cdot2^{j-1}-2 ,\\
\sum_{k= 3\, 2^{j-1}-1}^n \xx_k & \text{ if } 3\cdot 2^{j-1}-1 \le n\le 2^{j+1}-2,\\
\end{cases}
\]
where $j$ is the unique positive integer such that $2^j-1\le n\le 2^{j+1}-2$. Let us prove that $\BB_1$ is a semi-normalized basis for $\YY$. Define $(\yy_n^*)_{n=1}^\infty$ by
\[
\yy_n^*=
\begin{cases}
\xx_n^* & \text{ if } 2^{j}-1 \le n\le 3\cdot 2^{j-1}-2 ,\\
 \xx_n^*- \xx_{n+1}^* & \text{ if } 3\cdot 2^{j-1}-1 \le n\le 2^{j+1}-3,\\
 \xx_n^* & \text{ if } n= 2^{j+1}-2,\\
\end{cases}
\]
where $j$ is as before. We have $\yy_n^*(\yy_k)=\delta_{k,n}$. Moreover,
\[
T_m:=S_m[\BB]=
\begin{cases}
S_m & \text{ if } 2^{j}-1 \le m\le 3\cdot 2^{j-1}-2 ,\\
S_m- \xx_{m+1}^*\otimes \xx_m & \text{ if } 3\,2^{j-1}-1 \le m\le 2^{j+1}-3,\\
 S_m & \text{ if } m= 2^{j+1}-2.\\
\end{cases}
\]
If we put
$b:=\sup_n \Vert \xx_n\Vert$, $d:=\sup_n\Vert \xx_n^*\Vert$ and let $K$ be  the basis constant of $\BB$, we have
\[
\sup_m\Vert T_m\Vert \le K+bd.\] It is clear that $[\yy_n \colon n\in\NN]=[\xx_n\colon n\in\NN]=\YY$. Hence, $\BB_1$ is basis for $\YY$. Finally, if we let $C=\sup_n \Vert \sum_{k=1}^n \xx_k\Vert$ we have  $\Vert \yy_n\Vert\le 2C$ and $ \Vert \yy_n^*\Vert\le 2d$ for all $n\in\NN$. Therefore, $\BB_1$ is semi-normalized.

Given a positive integer $n$ there are unique integers $j=j(n)$ and $k=k(n)$ with $j\ge 1$ and $1\le k \le 2^{j-1}$ such that 
\[
n=\begin{cases} 2^{j}+2k-3& \text{ if $n$ is odd,}\\
 2^{j}+2k-2 & \text{ if $n$ is even.}\\
 \end{cases}
\]
 Define $\BB_2=(\zz_n)_{n=1}^\infty$ in $\YY$ and $(\zz_n^*)_{n=1}^\infty$ in $\YY^*$ by
\begin{align*}
\zz_n&=
\begin{cases}
\yy_{2^j-2+k} +\yy_{3\,2^{j-1}-2+k} & \text{if $n$  is odd,}\\
\yy_{2^j-2+k} -\yy_{3\,2^{j-1}-2+k} & \text{if $n$ is even,}
\end{cases}\\
\zz_n^*&=
\begin{cases}
\frac{1}{2} \left(\yy^*_{2^j-2+k} +\yy^*_{3 \, 2^{j-1}-2+k} \right) & \text{if $n$ is odd,}\\
\frac{1}{2} \left(\yy^*_{2^j-2+k} -\yy^*_{3\, 2^{j-1}-2+k} \right) & \text{if $n$ is  even,}
\end{cases}
\end{align*}
where $j=j(n)$ and $k=k(n)$. We have $\zz_n^*(\zz_k)=\delta_{k,n}$. Moreover, if $j=j(m)$ and $k=k(m)$,
\[
S_m[\BB_2]=R_{j,k}:=T_{2^j+k-2} + T_{3\, 2^{j-1}+k-2}-T_{3\, 2^{j-1}-2}
\] 
for $m$ even, and
\[
S_m[\BB_2]
=R_{j,k-1} + \frac{1}{2} \sum_{\varepsilon,\delta\in\{2,3\}}  \yy_{\varepsilon2^{j-1}+k-2}^*\otimes \yy_{\delta 2^{j-1}+k-2}
\]
for $m$ odd. We infer that 
\[\sup_m \Vert S_m[\BB_2]\Vert \le3 K+3bd+2 d C.\] It is clear that $[\zz_n \colon n\in\NN]= \YY$. Hence, $\BB_2$ is a basis for $\YY$. 

Let $m\in\NN$, $m\ge 2$. Choose $j\in\NN$ such that $2^{j+1}-2\le m \le 2^{j+2}-3$ and
put
\begin{align*}
f&=\sum_{n=2^{j}-1}^{2^{j+1}-2} \zz_n =2 \sum_{n=2^j-1}^{3\, 2^{j-1}-2} \yy_n=2 
\sum_{n=2^j-1}^{3\, 2^{j-1}-2} \xx_n,
\text{ and }\\
g&=\sum_{n=2^{j}-1}^{2^{j+1}-2} (-1)^{n-1}\zz_n =2 \sum_{n=3\, 2^{j-1}-1}^{2^{j+1}-2} \yy_n
=2 \sum_{k=3\, 2^{j-1}-1}^{2^{j+1}-2} (2^{j+1}-k-1) \xx_k.
\end{align*}
We have $\Vert f\Vert\le 4C$ and \[
\frac{m+3}{4} \le 2^{j}=\xx_{3\, 2^{j-1}-1}^*(g)\le d \Vert g\Vert.
\]
Since $(f+g)=2S_A(f)$, where $A=\{ 2^{j}+2k-3 \colon 1\le k\le 2^{j-1}\}$,
\[
L_{m}[\BB_2]\ge \frac{\Vert S_A(f)\Vert }{\Vert f\Vert} \ge \frac{\Vert g \Vert -\Vert f\Vert } {2 \Vert f\Vert}\ge
 \frac{m+3}{32dC}-\frac{1}{2}.
\]
Consequently, $L_m[\BB_2]\approx m$ for $m\in\NN$. 

(b) $\Rightarrow$ (c) Let $\YY$ be a Banach space finitely representable in $\XX$ and let $\BB$ be a basis for $\YY$ with $L_m[\BB]\approx m$ for $m\in\NN$. Applying the GOW-method to $\BB$ yields a quasi-greedy basis $\GOW(\BB)$ for $\YY\oplus \ell_2$ with $L_m[\GOW(\BB)]\approx \log m$ for $m\ge 2$. Using 
Lemma~\ref{BasicSequenceLemma} and Dvoretzky's theorem (see e.g. \cite{AlbiacKalton2016}*{Theorem 13.3.7}) we get that $\YY\oplus\ell_2$ is crudely finitely representable in $\XX$. Thus a suitable renorming of $\YY\oplus\ell_2$ is finitely representable in $\XX$ (see e.g. \cite{AlbiacKalton2016}*{Proposition 12.1.13}).

 (c) $\Rightarrow$ (d)  is obvious (d) $\Rightarrow$ (a) is a consequence of Theorem~\ref{LogEstimateSuperR}.
 \end{proof}

Before going on, for the sake of expositional ease we re-state a result of James adapted to our interests.

\begin{theorem}[see \cite{James1972}*{Theorem 2 and Theorem 3}]\label{JamesTheorem}
 Let $0<a\le b<\infty$ and $K\ge 1$  and  suppose $\XX$ is a superreflexive Banach space. There are $0<C\le D<\infty$ and $1<p\le q<\infty$ such that  for every basic sequence $(\xx_n)_{n=1}^\infty$ in $\XX$  with basis constant at most $K$ and $a \le \Vert \xx_n\Vert \le b$ for all $n\in\NN$
we have
 \begin{equation}\label{JamesBounds}
C\left(\sum_{n=1}^\infty |a_n|^q \right)^{1/q }\le \Vert f \Vert \leq D 
\left(\sum_{n=1}^\infty |a_n|^p \right)^{1/p } 
\end{equation}
 for any $f\in [\xx_n \colon n\in \NN]$.
\end{theorem}

Notice that  if  a basis $\BB= (\xx_n)_{n=1}^\infty$ verifies \eqref{JamesBounds} then it is verified that
\[
Cm^{1/q} \textstyle \le \Vert \sum_{n=1}^m \xx_n\Vert \le D m^{1/p},\quad \forall m\in\NN
\]
 \text{ and } 
 \[
k_m [\BB] \le \frac{D}{C} m^{1/p-1/q},\quad \forall m\in\NN.
\]

{Theorem~\ref{CharacterizationSuperR3} is just a part  of the following.

\begin{theorem}\label{MainTheorem2} Let $\XX$ be a Banach space. The following are equivalent:

\begin{itemize}

\item[(a)] $\XX$ is non-superreflexive.

\item[(b)] There is a basic sequence $\BB$ in $\XX$ with 
 $
 L_m[\BB]\approx m \text{ for }m\in\NN.
 $

\item[(c)] There is a basic sequence $\BB$ in $\XX$ with 
 $
 k_m[\BB]\approx m \text{ for }m\in\NN.
 $
\end{itemize}
\end{theorem}

 \begin{proof}
 Before proving  (a) $\Rightarrow$ (b) we realize that  a careful reading of the proof of Theorem~\ref{MainTheorem} reveals that the constants related to the bases obtained there are independent of  the particular Banach space $\XX$ we are dealing with. In particular, focussing on item (b), we claim that there are universal constants $C$ and $K$  so that for any non-superreflexive Banach space $\XX$ there is a basis $\BB$ with basis constant at most $K$ for a Banach space finitely representable in $\XX$ such that $L_m[\BB]\ge Cm$ for all $m\in\NN$. 
 
Pick a sequence $(\varepsilon_j)_{j=1}^\infty$  with  $0<\varepsilon_j<1$ and $\lambda:=\prod_{j=1}^\infty (1+\varepsilon_j)<\infty$. We claim that there exist
 $(\BB_j,\XX_j)_{j=1}^\infty$ such that 
 \begin{itemize}
 \item[(i)] $\BB_j=(\yy_{j,n})_{n=1}^{2^{j-1}}$ is a finite basic sequence in $\XX$  with basis constant at most $2K$, 
  \item[(ii)] for each $j\in\NN$ there exist a  vector  $f_j\in[\yy_{j,n} \colon 1\le n \le  2^{j-1}]$ and a set $A_j\subseteq\{1,\dots, 
 2^{j-1} \}$ so that 
  $
  \Vert S_{A_j}(f_j)\Vert  >  C2^{j-2}\Vert f_j\Vert.
  $
  \item[(iii)]  $\XX_j$ is a non-superreflexive subspace of $\XX$, and
  \item[(iv)] if we put $\YY_j=[\yy_{k,n} \colon 1\le k \le j, 1\le n \le 2^{k-1}]$, we have 
  \[
  \Vert f \Vert \le (1+\varepsilon_j) \Vert f+g \Vert \text{ for all } f\in\YY_j \text{ and }g\in\XX_j.
  \]
  \end{itemize}
To see this we proceed recursively.  Put   $\YY_0=\{0\}$, $\XX_0=\XX$, and for $j\in \NN$ assume that $(\BB_{k},\XX_{k})$ has been constructed for $1\le k \le j-1$ (nothing is constructed in the case $j=0$).  There is a basis $\BB_{j}^{\prime}=(\zz_n)_{n=1}^\infty$ with basis constant at most $K$  for a Banach space finitely representable in $\XX_{j-1}$ and  with $L_m[\BB_{j}^{\prime}]\ge Cm$ for all $m$. Consequently, there are $g_j\in[ \zz_n \colon 1\le n \le 2^{j-1}]$  and  $A_j\subseteq \{1,\dots,2^{j-1}\}$ with 
\[
 \Vert S_{A_j}(g_j)\Vert > \frac{1+\varepsilon_j}{2} C 2^{j-1} \Vert g_j\Vert.
 \]
There is an isomorphic embedding $T_j\colon[\zz_n \colon 1\le n\le 2^{j-1}]\to \XX$ such that $\Vert T_j\Vert \Vert T_j^{-1}\Vert \le 1+\varepsilon_j$.
The sequence $(\yy_n)_{n=1}^{2^{j-1}}$, where  $\yy_n=T_j(\zz_{n})$ has properties (i) and (ii). By Lemma~\ref{BasicSequenceLemma}, there is a finite co-dimensional subspace $\XX_j$ of $\XX$ verifying (iv). Since $\XX$ is not superreflexive, neither is $\XX_j$.

Next, let $\BB=(\xx_n)_{n=1}^\infty$ be the sequence obtained by putting in a row the bases $(\BB_j)_{j=1}^\infty$. Let us show that $\BB$ is a basic sequence. Notice that the length of the first $j$ blocks that form $\BB$ is 
$\sum_{k=1}^{j} 2^{k-1}=2^j-1$.
Let  $(a_n)_{n=1}^\infty$ be a sequence of scalars eventually zero and fix $m\in\NN$. Pick  $j\in\NN$ and $0\le k\le 2^{j-1}-1$   determined by the condition  $m= 2^{j-1} + k$.  We have (with the convention that $\sum_1^0 a_{n} x_n=0$)
\begin{align*}
\left\Vert \sum_{n=1}^m a_n \xx_n\right\Vert&\le  \left\Vert \sum_{n=1}^{2^{j-1}-1} a_n \xx_n\right\Vert+  \left\Vert\sum_{n=2^{j-1}}^{2^{j-1}+k} a_n \xx_n\right\Vert\\
&\le  \left\Vert \sum_{n=1}^{2^{j-1}-1} a_n \xx_n\right\Vert+ 2K  \left\Vert\sum_{n=2^{j-1}}^{2^j-1} a_n \xx_n\right\Vert\\
&\le  (1+2K)  \left\Vert \sum_{n=1}^{2^{j-1}-1} a_n \xx_n\right\Vert+ 2K  \left\Vert\sum_{n=1}^{2^j-1} a_n \xx_n\right\Vert\\
&\le  \lambda(1+4K)  \left\Vert \sum_{n=1}^{\infty} a_n \xx_n\right\Vert.
\end{align*}
Hence $\BB$ is a basic sequence with basis constant  at most $\lambda(1+4K)$. Let us estimate its conditionality constants. Let $m\in\NN$ and choose $j\in\NN$ such  
$2^j- 1 \le m \le 2^{j+1}-2$. We have 
\[
L_m[\BB] \ge  \frac{ \Vert S_{A_j}(f_j)\Vert }{\Vert f_j\Vert}\ge C2^{j-2}\ge \frac{C}{8}(m+2).
\]

(b) $\Rightarrow$ (c) is obvious, and (c) $\Rightarrow$ (a)  is an immediate consequence of Theorem~\ref{JamesTheorem}. 
 \end{proof}
 
Our next corollary is a  
 consequence of Theorem~\ref{MainTheorem} and Theorem~\ref{MainTheorem2}.

\begin{corollary}\label{MainTheorem3} Let $\XX$ be a Banach space. The following are equivalent:

\begin{itemize}

\item[(a)] $\XX$ is superreflexive.

\item[(b)] For  every semi-normalized  basic sequence $(\xx_n)_{n=1}^\infty$ in $\XX$ there are $0<C<\infty$ and $0<\alpha<1$ such that
\[
\Big\Vert \sum_{n=1}^m \xx_n\Big\Vert \le C m^\alpha, \quad m\in\NN.
\]

\item[(c)] For  every semi-normalized  basic sequence $(\xx_n)_{n=1}^\infty$ in $\XX$  there are $0<C<\infty$ and $0<\alpha<1$ such that

\[
 \Big\Vert \sum_{n=1}^m \xx_n\Big\Vert \ge C m^\alpha, \quad m\in\NN.
\]

\item[(d)] For  every semi-normalized  basic sequence $\BB$ in $\XX$  there are $0<C<\infty$ and $0<\alpha<1$ such that
\[
\textstyle k_m[\BB]  \le C  m^\alpha, \quad m\in\NN.
\]

\item[(e)]For every quasi-greedy basis $\BB$ of any Banach space finitely representable in $\XX$ there are $0<C<\infty$ and $0<\alpha<1$ such that
\[
\textstyle{k_m[\BB]\le C  (\log m)^\alpha,\quad m\ge 2}.
\]

\end{itemize}

\end{corollary}

\begin{proof}(a) $\Rightarrow$ (b), (a) $\Rightarrow$ (c) and (a) $\Rightarrow$ (d) follow from Theorem~\ref{JamesTheorem}, while
(a) $\Rightarrow$ (e)  follows from Theorem~\ref{LogEstimateSuperR}.
Clearly,  a basis as in (e) does not verify the condition (d) in Theorem~\ref{MainTheorem}. Hence the implication (e) $\Rightarrow$ (a) holds. In turn,  (d) $\Rightarrow$ (a) is a consequence of  Theorem~\ref{MainTheorem2}. Let us prove (b) $\Rightarrow$ (a) (the proof of  (c) $\Rightarrow$ (a) is similar). Notice that by Theorem~\ref{CharacterizationReflexivity} and  Lemma~\ref{PProperty}(c), there are universal constants $b$, $C$, $K\in(0,\infty)$ such that for any
non-superreflexive Banach space there is a basis $(\xx_n)_{n=1}^\infty$ with basis constant at most $K$ for a Banach space finitely representable in $\XX$ with  $\Vert \xx_n\Vert \le b$ for all $n\in\NN$ and
$|\sum_{n=1}^\infty a_n|\le C \Vert \sum_{n=1}^\infty \xx_n\Vert$ for any sequence of scalars $(a_n)_{n=1}^\infty$ eventually zero.

 Assume that $\XX$ is not superreflexive. Pick a sequence $(\varepsilon_j)_{j=1}^\infty$  with  $0<\varepsilon_j<1$ and $\lambda:=\prod_{j=1}^\infty (1+\varepsilon_j)<\infty$. As in the proof of Theorem~\ref{MainTheorem2} we recursively construct a sequence $(\BB_j,\XX_j)_{j=1}^\infty$ such that
 \begin{itemize}
 \item[(i)] $\BB_j=(\yy_{j,n})_{n=1}^{2^{j-1}}$ is a finite basic sequence in $\XX$  with basis constant at most $2K$,
  \item[(ii)]  $\Vert \yy_{j,n}\Vert\le 2b$ for $1\le n \le 2^{j-1}$,
  \item[(iii)] $ |\sum_{n=1}^{2^{j-1}} a_n|\le 2 C\Vert \sum_{n=1}^{2^{j-1}} \yy_{j,n}\Vert$ for any $2^{j-1}$-tuple
  $(a_n)_{n=1}^{2^{j-1}}$,
  \item[(iv)]  $\XX_j$ is a non-superreflexive subspace of $\XX$,  and
  \item[(v)] by letting $\YY_j=[\yy_{k,n} \colon 1\le k \le j, 1\le n \le 2^{k-1}]$ we have
  \[
  \Vert f \Vert \le (1+\varepsilon_j) \Vert f+g \Vert \text{ for all } f\in\YY_j \text{ and }g\in\XX_j.
  \]
  \end{itemize}
 Putting in a row the bases $(\BB_j)_{j=1}^\infty$ yields a   basic sequence $\BB=(\xx_n)_{n=1}^\infty$ with
 basis constant at most $\lambda(1+4K)$ verifying $(2C)^{-1}\le\Vert \xx_n\Vert\le2b$ for all $n\in\NN$.
 Moreover, for any $j\in\NN$,
\[
 2^{j}  \le  4C \left\Vert\sum_{n=2^{j-1}}^{2^j-1}  \xx_n\right\Vert  \le
8 \lambda C (1+4K) \left\Vert\sum_{n=1}^{2^j}  \xx_n\right\Vert.
\]
Therefore, (b) does not hold.
\end{proof}

\begin{remark} Let $\BB=(\xx_n)_{n=1}^\infty$ be a basic sequence with basis constant at most $K$ in a Banach space $\XX$ with $0<a\le \Vert \xx_n\Vert \le b<\infty$ for all $n\in\NN$. Then the constants $C$ and $\alpha$ in items (b), (c), (d) of Corollary~\ref{MainTheorem3} depend only on $a$, $b$, $K$ and the particular superreflexive Banach space
$\XX$. Similarly, since a quantitative version of Theorem~\ref{LogEstimateSuperR} is valid, for a quasi-greedy basis
 $\BB=(\xx_n)_{n=1}^\infty$  with quasi-greedy constant at most $\Gamma$ and  $0<a\le \Vert \xx_n\Vert \le b<\infty$ for all $n\in\NN$,  the constants  $C$ and $\alpha$ in item (e) depend only on $a$, $b$, $\Gamma$ and the particular  superreflexive Banach space we are dealing with. \end{remark}\begin{remark}Since the quasi-greedy basic sequence problem remains unsolved, it seems hopeless to characterize a Banach space $\XX$ being superreflexive in terms of the behavior of basic sequences in $\XX$. The situation is quite different if we consider ``finite'' basic sequences. Note that the greedy algorithm $(\GG_m)_{m=1}^\eta$ of a basic sequence $(\xx_n)_{n=1}^\eta$ can be defined even when $\eta$ is a natural number.
\end{remark}

We close with the aforementioned characterization of superreflexivity, which is now obvious with hindsight.

\begin{theorem}Let $\XX$ be a Banach space.
\begin{itemize}
\item[(a)]  $\XX$ is superreflexive if and only if for any $0<a\le b<\infty$ and any $K\in(0,\infty)$ there are $C\in(0,\infty)$
and $0<\alpha<1$ such that for any integer $\eta\ge 2$ and any basic sequence $\BB=(\xx_n)_{n=1}^\eta$ in $\XX$ with
quasi-greedy constant at most $K$ and
 $a\le\Vert \xx_n\Vert \le b$ for all $1\le n\le \eta$ we have $k_m[\BB]\le C (\log m)^\alpha$ for $2\le m\le \eta$.

 \item[(b)] $\XX$ is non-superreflexive if and only if there is $C\in(0,\infty)$ and a sequence $(\BB_j)_{j=1}^\infty$ of uniformly semi-normalized  finite basic sequences in $\XX$ with uniformly bounded quasi-greedy constants such that $\lim_j |\BB_j|=\infty$ and $k_m[\BB_j]\ge C \log m$ for any $j\in\NN$ and all $1\le m \le |\BB_j|$.

 \end{itemize}
 \end{theorem}

\subsection*{Acknowledgement} The authors would like to thank Prof.\ Thomas Schlum\-precht for helpful discussions and for bringing up to their attention  the paper \cite{PisierXu1987}.



\end{document}